\documentclass[11pt]{amsart}

\usepackage{amsmath,amssymb,amsthm}
\usepackage[colorlinks=true,linkcolor=blue,citecolor=blue,urlcolor=blue]{hyperref}

\theoremstyle{plain}
\newtheorem{theorem}{Theorem}
\newtheorem{lemma}{Lemma}

\newtheorem{case}{Case}
\theoremstyle{remark}
\newtheorem{remark}{Remark}

\title{Residue Restrictions for a Two-Color Partition Series}
\author{Aman Singh}
\address{Department of Mathematics, Indian Institute of Technology, Roorkee-
247667, Uttarakhand, India}
\email{amansingh9839269454@gmail.com}
\date{\today}
\begin{document}
\begin{abstract}
    We study residue restrictions for a two-color partition series $P(q)=(q^4;q^4)_\infty S(q)$ arising from work of Andrews and Bachraoui on partitions with odd smallest part. Motivated by the explicit exponent structure in the Bailey-transform formulas for the associated generating function, we obtain elementary restrictions on the support of the coefficients. In particular, we show that the residue class $4\pmod8$ does not occur, and we prove a further refinement modulo $16$, yielding several vanishing classes for the coefficients of the series. Our argument is completely residue-theoretic and avoids the use of modular completions.
\end{abstract}
\maketitle
\section{Introduction}
\noindent We use the standard $q$-series notation as for instance in~\cite[Chapter 1]{GasperRahman}
\[
(a;q)_0:=1,\qquad (a;q)_n:=\prod_{j=0}^{n-1}(1-aq^j),\qquad
(a;q)_\infty:=\prod_{j=0}^{\infty}(1-aq^j), \quad |q|<1.\]
For a formal power series
\[
P(q)=\sum_{n\ge 0} A(n)q^n, \qquad \text{with } N\ge 1,
\]
and \(0\le r < N\), the sieve operator is defined by
\[
P(q)\big|_{S_{N,r}}=
\sum_{\substack{n\ge 0\\ n\equiv r \,(\mathrm{mod}\, N)}} A(n)q^n.
\]
Andrews and Bachraoui ~\cite{AndrewsBachraoui} studied the series
\[
S(q):=\sum_{m\ge0}\frac{(q;q^2)_{m}^2q^{2m}}{(-q^2;q^2)_{m}}
\] in connection with two-color partitions with odd smallest part and conjectured the following Hecke-type formula
$$ \big((q^4;q^4)_\infty S(q)\big)|_{S_{2,0}}= \sum_{n\ge 0}(-1)^n q^{6n^2+4n}(1+q^{4n+2})\sum_{|j|\le n}(-1)^j q^{-2j^2}.$$
Banerjee and Bringmann \cite{BanerjeeBringmann} proved this conjecture by completing both sides to modular objects and then using Sturm's theorem.  At the end of their paper, they asked whether a more direct proof could be found, for example by using Bailey pairs, and they also suggested studying the odd residue classes. This was subsequently proved in ~\cite{AndrewsBachraoui2}, which yields the corollary $$P(q)|_{S_{4,1}}=0.$$ 
The Bailey-transform identities obtained in \cite{AndrewsBachraoui2}
naturally lead to questions concerning the residue support of the associated series $P(q)$. While Corollary~4 of
\cite{AndrewsBachraoui2} determines the odd residue behavior modulo $4$, the even residue classes appear to admit additional
restrictions. In this short note we record elementary consequences of the exponent structure arising from equations
(31) and (32) of \cite{AndrewsBachraoui2}. In particular, we show that the
class $4 \pmod 8$ does not occur and obtain a refinement modulo $16$.
\vspace{1mm}
\begin{theorem}~\label{main}
We have
\[
P(q)\big|_{S_{8,4}}=0.
\]
\end{theorem}
\begin{theorem}
We have
\[
[q^n]P(q)=0, \quad
n\equiv 1,4,5,6,9,12,13 \pmod{16}.
\]
\end{theorem}
\section{Preliminary Facts}
\noindent Andrews and Bachraoui considered a generalized version of $S(q)$ in ~\cite{AndrewsBachraoui} by adding in a parameter $a$ and defining 
$$S_a(q):=\sum_{m\ge 0}\frac{(aq;q^2)_m(q/a;q^2)_m q^{2m}}{(-q^2;q^2)_m}, \qquad P_a(q)=(q^4;q^4)_\infty S_a(q).$$
From ~\cite[Equation 32]{AndrewsBachraoui2}, we get 
$$P_a(q)=\sum_{R\ge 0}\alpha_R(a)\gamma_R=\sum_{h\in \mathbb{Z}} (-1)^h a^h q^{h^2} \gamma_{|h|},$$
where $$ \gamma_R=q^{2R}\sum_{s\ge 0}(-1)^s q^{6s^2+(6R+4)s}(1+q^{2R+4s+2}).$$
Substituting the expression of $\gamma_R$ above, we obtain
$$P_a(q)=\sum_{h\in \mathbb{Z}}\sum_{s=0}^\infty (-1)^h a^h q^{h^2+2|h|+6s^2+(6|h|+4)s}(1+q^{2|h|+4s+2}),$$
and now expanding the parenthesis, we get
$$E(|h|,s)=h^2+2|h|+6s^2+(6|h|+4)s,$$
$$F(|h|,s)=E(R,s)+2|h|+4s+2.$$
Set \(R=|h|\), we obtain
 \begin{equation}
     E(R,s)=R^2+2R+6s^2+(6R+4)s,
 \end{equation}
 \begin{equation}
    F(R,s)=R^2+4R+6s^2+(6R+8)s+2 .
 \end{equation}
\section{Residue Restrictions for 
\texorpdfstring{$E(R,s)$ and $F(R,s)$}{E(R,s) and F(R,s)}
Modulo \texorpdfstring{$8$}{8}}
\noindent In this section, we will prove that $E(R,s), F(R,s)\not \equiv 4 \pmod 8$.
 \begin{lemma}
 For $R\ge 0$ and $s\ge 0$
     $$E(R,s) \not \equiv 4\pmod{8}.$$
 \end{lemma}
 \begin{proof}
Let $R\equiv \rho \pmod4 $, $\rho \in\{0,1,2,3\}$ and we have,
$$E(R,s)=(R^2+2R)+6s^2+(6R+4)s.$$
Now, we compute \(R^2+2R \pmod 8\). Since,
\[
R(R+2)\equiv
\begin{cases}
0 \pmod{8}, & R\equiv 0 \pmod{2}\\[2mm]
3\text{ or }7 \pmod{8}, & R\equiv 1 \pmod{2},
\end{cases}
\]
and \(R^2+2R \pmod 8\) depends only on \(R \pmod 4\), we have
\[
\begin{array}{c|c}
R \bmod 4 & R^2+2R \bmod 8\\
\hline
0 & 0\\
1 & 3\\
2 & 0\\
3 & 7
\end{array}
\]
also,
\[
(R+4)^2+2(R+4)\equiv R^2+2R \pmod{8}.
\]
Now, consider the residues of $6s^2+cs \pmod{8},$
$$6R+4\equiv4,10,16,22 \equiv 4,2,0,6 \pmod{8}$$ for $R=0,1,2,3$ so the relevant values of $c$ are $0,2,4,6.$
\[
\begin{array}{c|cccc}
c \bmod 8 & 0 & 2 & 4 & 6 \\ \hline
6s^2+cs \bmod 8 & 0 & 6 & 0 & 6 \\
                & 0 & 0 & 4 & 4 \\
                & 0 & 2 & 0 & 2 \\
                & 0 & 4 & 4 & 0
\end{array}
\]
with
\[
s \equiv 0,1,2,3 \pmod{4}.
\]
Also,
\[
6(s+4)^2+c(s+4)
\equiv 6s^2+48s+96+cs+4c
\equiv 6s^2+cs \pmod{8}.
\]
Thus, the expression is periodic modulo \(4\) in \(s\).
Now let,
\[
R\equiv \rho\pmod{4}, \quad \rho\in\{0,1,2,3\},
\]
and consider the following cases:
\begin{case}
 \(\rho\equiv 0 \pmod{4}\Rightarrow R\equiv 0 \pmod{4}\) \(\Rightarrow R\equiv 0 \pmod{2}\).
Thus,
\[
R^2+2R \equiv 0 \pmod{8}
\quad \text{and} \quad
6R+4 \equiv 4 \pmod{8}.
\]

With \(c=4\),
\[
6s^2+4s \equiv
\begin{cases}
0 \pmod 8, & \text{if } s \text{ is even},\\[2mm]
2 \pmod 8, & \text{if } s \text{ is odd}.
\end{cases}
\]
Indeed, if \(s=2k+1\), then
\[
6s^2+4s = 6(2k+1)^2 + 4(2k+1)
= 8(3k^2+4k+1)+2 \equiv 2 \pmod 8.
\]
Thus
\[
E \equiv \{0, 2\} \pmod 8.
\]
\end{case}
\begin{case}
$\rho \equiv 1 \pmod 4 \Rightarrow R\equiv 1 \pmod{4}.$
So,
\[R^2+2R \equiv 3 \pmod{8}\]
\[ \text{and }\;\;
6R+4 \equiv 6\cdot 1+4 = 10 \equiv 2 \pmod{8}.
\]
With \(c=2\),
\[
6s^2+2s = 2s(3s+1).
\]
Evaluating \(s=4k+i\), where \(i=0,1,2,3\), we get
\begin{align*}
s &\equiv 0 \pmod{4}: \qquad 2(4k)(12k+1) = 8k(12k+1) \equiv 0 \pmod{8}, \\
s &\equiv 1 \pmod{4}: \qquad 2(4k+1)(12k+4) \equiv 0 \pmod{8}, \\
s &\equiv 2 \pmod{4}: \qquad 2(4k+2)(12k+7) \equiv 4 \pmod{8}, \\
s &\equiv 3 \pmod{4}: \qquad 2(4k+3)(12k+10) \equiv 4 \pmod{8}.
\end{align*}
So,
\[
6s^2+2s \equiv 0 \pmod{8}
\quad \text{for } s\equiv 0,1 \pmod{4},
\]
and
\[
6s^2+2s \equiv 4 \pmod{8}
\quad \text{for } s\equiv 2,3 \pmod{4}.
\]
Hence,
\[
E \equiv 3+\{0,4\}\equiv \{3,7\}\pmod{8}.
\]
\end{case}
\begin{case}
$\rho\equiv2\pmod 4 \implies R\equiv 2\pmod 4.$ \text{So,}
\[R^2+2R \equiv 0 \pmod{8}\; 
\text{ and}\;\; 6R+4 \equiv 0 \pmod{8}.
\]
For \(c=0\),
\[
6s^2 \equiv 0 \pmod{8} \quad \text{for } s \text{ even},
\]
and
\[
6s^2 \equiv 6 \pmod{8} \quad \text{for } s \text{ odd}.
\]
Hence
\[
E \equiv 0 + \{0,6\}\equiv \{0,6\} \pmod{8}.
\]
\end{case} 
\begin{case}    
 $\rho\equiv 3 \pmod 4 \implies R \equiv 3 \pmod 4.$ Therefore,
\[
R^2+2R \equiv 7 \pmod{8}
\text{ and }
6R+4 \equiv 6 \pmod{8}.
\]
Taking \(c=6\) we get,
\[
6s^2+6s = 6s(s+1).
\]
Since \(s(s+1)\) is even, writing
\[
s(s+1)=2t,\qquad t\in \mathbb{N}.
\]
Thus,
\[
6s(s+1)=12t \equiv 4t \pmod{8}.
\]
Now
\[
t=\frac{s(s+1)}{2},
\]
and the values depend on \(s \pmod{4}\):
\[
\begin{array}{c|c}
s \pmod{4} & \dfrac{s(s+1)}{2} \pmod{2} \\ \hline
0,3 & 0 \\
1,2 & 1
\end{array}
\]
Therefore,
\[
12t \equiv 0 \pmod{8} \quad \text{if } s\equiv 0,3 \pmod{4},
\]
and
\[
12t \equiv 4 \pmod{8} \quad \text{if } s\equiv 1,2 \pmod{4}.
\]
So
\[
E \equiv 7+\{0,4\}=\{7,3\}\pmod{8},
\]\end{case}
\noindent Hence, by the above four cases we can conclude that\[
E\not \equiv 4 \pmod{8}.
\]
\end{proof}
\setcounter{case}{0}
\begin{lemma}
For $R\ge 0$ and $s\ge 0$
\[
F(R,s)=R^2+4R+6s^2+(6R+8)s+2 \not\equiv 4 \pmod{8}.
\]
\end{lemma}
\begin{proof}
We have,
\[
F(R,s)=\bigl(R^2+4R+2\bigr)+6s^2+(6R+8)s
\]
and
\[
R^2+4R+2=(R+2)^2-2,
\]
which depends only on \(R \pmod 4\).
\[
\begin{array}{c|cccc}
R \bmod 4 & 0 & 1 & 2 & 3 \\ \hline
R^2+4R+2 \bmod 8 & 2 & 7 & 6 & 7
\end{array}
\]
Now, consider the following cases:
\begin{case}
$R\equiv0 \pmod 4$
$$R^2+4R+2 \equiv 2 \pmod{8}$$ 
and 
\[6R+8 \equiv0\pmod{8}.\]
So taking  c=0,
\[
6s^2 \equiv 0,6 \pmod{8}
\]
\text{according as } $s$ \text{is even or odd. Hence,}
\[
F \equiv 2+\{0,6\}=\{2,0\} \pmod{8}.
\]
\end{case}
\begin{case}
$R\equiv 1 \pmod 4$
$$R^2+4R+2 \equiv 7 \pmod{8}$$
and
\[
6R+8 \equiv 14 \equiv 6 \pmod{8}.
\]
\text{So taking } c=6
\[
6s^2+6s=6s(s+1).
\]
Similar to the case for $\rho\equiv 3 \pmod 4$ in Lemma 1,
\[
F \equiv 7+\{0,4\}=\{7,3\} \pmod{8}.
\]
\end{case}
\begin{case}
$R \equiv 2 \pmod4$
$$R^2+4R+2 \equiv 6 \pmod{8}$$
\[
6R+8 \equiv 20 \equiv 4 \pmod{8}.
\]
\text{So taking } c=4
and proceeding similar to $\rho \equiv 0 \pmod 4$ in Lemma 1
\[
F \equiv 6+\{0,2\}=\{6,0\} \pmod{8}.
\]
\end{case} 
\begin{case}
$R\equiv 3 \pmod 4$
$$R^2+4R+2 \equiv 7 \pmod{8}$$
\[
6R+8 \equiv 26 \equiv 2 \pmod{8}
\]
\text{So taking } c=2
and proceed similar to $\rho \equiv 1 \pmod 4$ in Lemma 1
\[
F \equiv 7+\{0,4\}=\{7,3\} \pmod{8}.
\]
\end{case}
\noindent Hence, by the above four cases we can conclude that\[
F\not \equiv 4 \pmod{8}.
\]
\end{proof}
\section{Proof of Vanishing Theorem}
\noindent Summarizing the residue classes, we obtain
\vspace{0.5mm}
\[
\begin{array}{c|c|c|c}
R \bmod 4 & E(R,s)\bmod 8 & F(R,s)\bmod 8 & \text{corresponding classes} \\ \hline
0 & \{0,2\} & \{0,2\} & S_{8,0},\, S_{8,2} \\
1 & \{3,7\} & \{3,7\} & S_{8,3},\, S_{8,7} \\
2 & \{0,6\} & \{0,6\} & S_{8,0},\, S_{8,6} \\
3 & \{3,7\} & \{3,7\} & S_{8,3},\, S_{8,7}
\end{array}
\]
\vspace{0.5mm}
\noindent From the expansion
\[
P(q)=\sum_{h\in\mathbb Z}\sum_{s\ge 0}(-1)^{h+s}
q^{E(|h|,s)}\bigl(1+q^{2|h|+4s+2}\bigr),
\]
every exponent occurring in P(q) is of the form $E(R,s)$ or $F(R,s)$
\[
q^{E(R,s)} \quad \text{or} \quad q^{F(R,s)},
\]
where \(R=|h|\ge 0\) and
\[
F(R,s)=E(R,s)+2R+4s+2.
\]
Hence, if \([q^m]P(q)\neq 0\), then there exist \(R,s\ge 0\) such that
\(m=E(R,s)\) or \(m=F(R,s)\).
By Lemma 1 and Lemma 2, neither \(E(R,s)\) nor \(F(R,s)\) can be congruent to \(4 \pmod{8}\).
Therefore,
\[
[q^m]P(q)=0
\qquad \text{for all } m\equiv 4 \pmod{8}.
\]
\[
\text{It follows } P(q)\big|_{S_{8,4}}=0.
\]
\begin{remark}
The proof is purely residue-theoretic. Once the exponents \(E(R,s)\) and
\(F(R,s)\) are shown to avoid the class \(4 \pmod 8\), the vanishing of
\([q^m]P(q)\) for \(m\equiv 4 \pmod 8\) follows immediately.
\end{remark}
\section{A modulo 16 refinement}
\noindent We keep the notation
\[
E(R,s)=R^2+2R+6s^2+(6R+4)s,
\qquad
F(R,s)=R^2+4R+6s^2+(6R+8)s+2.
\]

\begin{lemma}
If \(R\) is odd, then
\[
E(R,s)\equiv F(R,s)\equiv 3 \pmod 4.
\]
\end{lemma}

\begin{proof}
Write \(R=2m+1\). Then
\[
R^2+2R=(2m+1)^2+2(2m+1)=4m^2+8m+3\equiv 3 \pmod 4.
\]
Also, since \(R\) is odd,
\[
6R+4\equiv 2 \pmod 4.
\]
Hence
\[
6s^2+(6R+4)s \equiv 2s^2+2s = 2s(s+1)\equiv 0 \pmod 4,
\]
because \(s(s+1)\) is always even. Therefore
\[
E(R,s)\equiv 3 \pmod 4.
\]
Finally,
\[
F(R,s)-E(R,s)=2R+4s+2\equiv 2+0+2\equiv 0 \pmod 4,
\]
so \(F(R,s)\equiv 3 \pmod 4\) as well.
\end{proof}

\begin{lemma}
Let \(R\equiv \rho \pmod 8\) and \(s\equiv \sigma \pmod 8\). Then
\[
E(R,s)\equiv E_{\rho,\sigma} \pmod{16},
\]
where the residues \(E_{\rho,\sigma}\) are listed in the following table
\[
\begin{array}{c|cccccccc}
\rho\backslash \sigma & 0 & 1 & 2 & 3 & 4 & 5 & 6 & 7\\ \hline
0 & 0 & 10 & 0 & 2 & 0 & 10 & 0 & 2\\
1 & 3 & 3 & 15 & 7 & 11 & 11 & 7 & 15\\
2 & 8 & 14 & 0 & 14 & 8 & 14 & 0 & 14\\
3 & 15 & 11 & 3 & 7 & 7 & 3 & 11 & 15\\
4 & 8 & 10 & 8 & 2 & 8 & 10 & 8 & 2\\
5 & 3 & 11 & 15 & 15 & 11 & 3 & 7 & 7\\
6 & 0 & 14 & 8 & 14 & 0 & 14 & 8 & 14\\
7 & 15 & 3 & 3 & 15 & 7 & 11 & 11 & 7
\end{array}
\pmod{16}.
\]
\end{lemma}

\begin{proof}
Write \(R=8u+\rho\) and \(s=8v+\sigma\). Since we work modulo \(16\),
\[
(8u+\rho)^2+2(8u+\rho)\equiv \rho^2+2\rho \pmod{16},
\]
and
\[
6(8v+\sigma)^2+(6(8u+\rho)+4)(8v+\sigma)
\equiv 6\sigma^2+(6\rho+4)\sigma \pmod{16}.
\]
Therefore
\[
E(R,s)\equiv \rho^2+2\rho+6\sigma^2+(6\rho+4)\sigma \pmod{16}.
\]
Evaluating this expression for \(\rho,\sigma\in\{0,1,\dots,7\}\) gives the table above.
\end{proof}

\begin{lemma}
Let \(R\equiv \rho \pmod 8\) and \(s\equiv \sigma \pmod 8\). Then
\[
F(R,s)\equiv F_{\rho,\sigma} \pmod{16},
\]
where the residues \(F_{\rho,\sigma}\) are listed in the following table
\vspace{0.3mm}
\[
\begin{array}{c|cccccccc}
\rho\backslash \sigma & 0 & 1 & 2 & 3 & 4 & 5 & 6 & 7\\ \hline
0 & 2 & 0 & 10 & 0 & 2 & 0 & 10 & 0\\
1 & 7 & 11 & 11 & 7 & 15 & 3 & 3 & 15\\
2 & 14 & 8 & 14 & 0 & 14 & 8 & 14 & 0\\
3 & 7 & 7 & 3 & 11 & 15 & 15 & 11 & 3\\
4 & 2 & 8 & 10 & 8 & 2 & 8 & 10 & 8\\
5 & 15 & 11 & 3 & 7 & 7 & 3 & 11 & 15\\
6 & 14 & 0 & 14 & 8 & 14 & 0 & 14 & 8\\
7 & 15 & 7 & 11 & 11 & 7 & 15 & 3 & 3
\end{array}
\pmod{16}.
\]
\end{lemma}

\begin{proof}
Write \(R=8u+\rho\) and \(s=8v+\sigma\). Modulo \(16\),
\[
(8u+\rho)^2+4(8u+\rho)+2\equiv \rho^2+4\rho+2 \pmod{16},
\]
and
\[
6(8v+\sigma)^2+(6(8u+\rho)+8)(8v+\sigma)
\equiv 6\sigma^2+(6\rho+8)\sigma \pmod{16}.
\]
Hence
\[
F(R,s)\equiv \rho^2+4\rho+2+6\sigma^2+(6\rho+8)\sigma \pmod{16}.
\]
Evaluating the right-hand side for \(\rho,\sigma\in\{0,1,\dots,7\}\) yields the table.
\end{proof}

\begin{theorem}
We have
\[
[q^n]P(q)=0
\qquad
(n\equiv 1,4,5,6,9,12,13 \pmod{16}).
\]
\end{theorem}

\begin{proof}
Every monomial occurring in \(P(q)\) is of the form \(q^{E(R,s)}\) or
\(q^{F(R,s)}\), where \(R=|h|\ge 0\) and \(s\ge 0\). Therefore, if
\([q^n]P(q)\neq 0\), then \(n=E(R,s)\) or \(n=F(R,s)\) for some \(R,s\).
If \(R\) is odd, then by the previous lemma,
\[
E(R,s)\equiv F(R,s)\equiv 3 \pmod 4,
\]
the odd residue classes that can occur are
\[
3,7,11,15.
\]
Thus no coefficient can occur in the classes
\[
1,5,9,13 \pmod{16}.
\]
Now suppose \(R\) is even. Then \(R\bmod 8\in\{0,2,4,6\}\). From the tables above, the possible residues of \(E(R,s)\) and \(F(R,s)\) modulo \(16\) are
\[
0,2,8,10,14.
\]
In particular, the classes $4,6,12 \pmod {16}$ do not occur.
Combining the odd and even cases, no exponent occurring in \(P(q)\) can lie in
\[
1,4,5,6,9,12,13 \pmod{16}.
\]
Hence
\[
[q^n]P(q)=0
\qquad
(n\equiv 1,4,5,6,9,12,13 \pmod{16}).
\]
\end{proof}
\begin{remark}
The modulo \(8\) and modulo \(16\) vanishing results show that the support of
\(P(q)\) is highly restricted in low powers of \(2\).
\end{remark}
\begin{remark}
The same residue-support argument applies to \(P_a(q)\) as well.
\end{remark}

\section{Conclusion}
\noindent In this paper we recorded elementary residue restrictions arising from the Bailey-transform decomposition of Andrews and Bachraoui. Using the explicit exponent structure, we showed that the residue class $4 \pmod 8$ does not occur in the support of
\[
P(q):=(q^4;q^4)_\infty S(q),
\]
and obtained a refinement modulo $16$. The argument is entirely residue-theoretic and suggests that further support restrictions may be derived by similar methods. These results clarify the residue behavior of the two-color partition series and complement the earlier work of Andrews, Bachraoui, Banerjee, and Bringmann. It would be interesting to formulate a more general residue-support theorem for related Bailey-transform identities.
\section{Acknowledgment}
\noindent The author is grateful to Professor Rahul Kumar for suggesting this work and for reading an early version of the draft. The author was supported by the grant ANRF/ECRG/2024/003222/PMS of the Anusandhan National
Research Foundation (ANRF), Govt. of India as an ANRF project student.

\end{document}